\documentclass[a4paper,12pt,reqno]{amsart}

\usepackage[utf8]{inputenc}
\usepackage{latexsym}
\usepackage{amssymb}
\usepackage{amsmath}
\usepackage{enumerate}
\usepackage{xcolor}
\usepackage{url}

\newtheorem{theorem}{Theorem}[section]
\newtheorem{lemma}[theorem]{Lemma}
\newtheorem{proposition}[theorem]{Proposition}
\newtheorem{corollary}[theorem]{Corollary}

\newtheorem{remark}[theorem]{Remark}
\newtheorem{definition}[theorem]{Definition}

\newcommand\re{\operatorname{Re}}
\newcommand\im{\operatorname{Im}}

\title[Isometries between  groups of invertible elements in $B(G)$]{Isometries between groups of invertible elements in Fourier-Stieltjes algebras}
\author{
Osamu~Hatori
}
\address{
Institute of Science and Technology,
Niigata University, Niigata 950-2181, Japan.
}
\email{hatori@math.sc.niigata-u.ac.jp
}

\author{
Shiho Oi
}
\address{
Department of Mathematics, Faculty of Science, 
Niigata University, Niigata 950-2181, Japan.
}
\email{shiho-oi@math.sc.niigata-u.ac.jp
}


\begin{document}

\maketitle\textbf{}

\begin{abstract}
We prove that 
if open subgroups of the groups of invertible elements in two Fourier-Stieltjes algebras are isometric as metric spaces, then the underlying locally compact groups are topologically isomorphic.
We describe the structure of isometric real algebra isomorphisms between Fourier-Stieltjes algebras
and apply it to prove the above result. 
\end{abstract}

\maketitle

\section{Introduction}


There has been a long-standing effort to determine the conditions of a linear subspace $A$ of $C(X)$, a space of complex-valued continuous functions on a compact Hausdorff space $X$, to ensure that $X$ can be uniquely determined. The Banach-Stone theorem is the most prominent result in this field. According to the theorem, in addition to $A=C(X)$, the crucial requirement is that the structure is a Banach space.

This paper focuses on how the Fourier-Stieltjes algebra $B(G)$ of a locally compact group $G$ encodes its structure. The algebra  $B(G)$ is defined as the complex linear span of the set of all continuous positive-definite functions on $G$.
It 
is a semisimple commutative Banach algebra of all matrix coefficients of continuous unitary representations of $G$.

By \cite[Theorem 20.20]{HS} (cf. \cite{spronk}) and \cite[Section 11.1 Exercise 9]{folland}, we identify $L^1(G)^*$ with $L^\infty(G)$ for a locally compact group $G$.
Therefore
the Fourier-Stieltjes algebra $B(G)$ is identified with the dual of the group $C^{*}$-algebra $C^{*}(G)$.  The dualities determine the norm on $B(G)$.  The enveloping von Neumann algebra of $C^{*}(G)$, denoted as $W^{*}(G)$, is identified with $B(G)^{*}$ when $B(G)$ is identified with the dual of $C^*(G)$. 
Throughout the paper, we assume that $G$ and $H$ are locally compact groups, with $e_1$ and $e_2$ denoting their respective units.

A celebrated theorem by Walter \cite{Wal} states that if $S\colon B(G)\to B(H)$ is an isometric complex algebra isomorphism, then $G$ and $H$ are topologically isomorphic as locally compact groups.
 However, it is impractical that the structure of $B(G)$ as a Banach space does not completely encode the group structure of $G$. For instance, consider two discrete abelian groups with the same cardinality, $\Gamma_1$ and $\Gamma_2$. Any bijection $\alpha\colon \Gamma_2\to \Gamma_1$ induces a surjective complex linear isometry from $L^1(\Gamma_1)\to L^1(\Gamma_2)$ by $f\mapsto f\circ\alpha$ for $f\in L^1(\Gamma_1)$, which in turn induces the complex linear isometry from $B(\widehat{\Gamma_1})\to B(\widehat{\Gamma_2})$ through the Fourier transform. However, it is possible for $\Gamma_1$ and $\Gamma_2$ to be nonisomorphic as groups; for example, if $\Gamma_1=\mathbb{Z}$ (the integer group) and $\Gamma_2=\mathbb{Z}\times \mathbb{Z}$, $B(\widehat{\Gamma_1})$ and $B(\widehat{\Gamma_2})$ are isometrically isomorphic as Banach spaces, but the group $\widehat{\Gamma_1}$ is not isomorphic to $\widehat{\Gamma_2}$.

On the other hand, in \cite{H3}, the first author showed that the metric structures of open subgroups of groups of invertible elements in the measure algebras on locally compact {\it abelian} groups guarantee the structure of the underlying locally {\it abelian} groups. Investigating whether the same result holds true for noncommutative groups would be interesting and natural. 
In this paper, we will provide an effective answer to this question for any locally compact group. The key is to characterize all isometric real algebra isomorphisms between two Fourier-Stieltjes algebras on locally compact groups, which is Theorem \ref{realisometryonBG}. Since any complex algebra isomorphism occurs at the same time as a real algebra isomorphism, Theorem \ref{realisometryonBG} is a generalization of the theorem proposed by Walter\cite{Wal}. We deduce that every isometric real algebra isomorphism between two Fourier-Stieltjes algebras automatically induces a topological group isomorphism between underlying groups. This characterization makes it possible to answer the above 
 question: Corollary \ref{topologicalygroupiso}.


\section{Isometric real algebra isomorphisms}
In this section, we consider surjective real linear isometries that also preserve the multiplication between $B(G)$ and $B(H)$.
To distinguish them from usual isometric isomorphisms (which are assumed to be complex linear), we refer to them as isometric real algebra isomorphisms. 
We present Theorem \ref{realisometryonBG}
 which characterizes an isometric real algebra isomorphism. This is the result of a comparison with a celebrated theorem of Walter  \cite{Wal}.
Authors such as Cohen \cite{Co}, Daws \cite{Daws}, Ilie \cite{Il}, Ilie and Spronk \cite{IlieSpronk}, and Le Pham \cite{LePham} have extensively studied complex linear homomorphisms on Fourier-Stieltjes algebras and Fourier algebras.
Our study presents a new approach for obtaining a representation of real linear operators on Fourier-Stieltjes algebras. 
In contrast to Walter's work, our results demonstrate that the assumption of real linearity is sufficient for the isomorphism of underlying groups for similar mappings.
Notably, every isometric real algebra isomorphism is either complex linear or conjugate-linear, which is an interesting finding.

Let $T: B(G) \to B(H)$ be an isometric real algebra isomorphism. We define $T^{*}: B(H)^{*} \to B(G)^{*}$ by 
\begin{equation}\label{dualmap}
T^{*}(\varphi)(f)=\re \varphi(Tf)-i \re\varphi(T(if)), \quad f \in B(G)
\end{equation}
for $\varphi\in B(H)^{*}$.
Then, $T^{*}$ is a surjective real linear isometry. 
For the convenience of the readers, we have included this proof. 
By a simple calculation, we see that $T^*(\varphi)$ is a complex linear functional on $B(G)$ for $\varphi\in B(H)^*$. 
We show that $\|T^*(\varphi)\|\le \|\varphi\|$ for every $\varphi\in B(H)^*$. Let $f\in B(G)$. There exists $0\le \theta<2\pi$ such that 
\[
T^*(\varphi)(f)=|T^*(\varphi)(f)|e^{i\theta}.
\]
Then, 
\begin{equation*}
\begin{split}
    |T^*(\varphi)(f)|&=e^{-i\theta}T^*(\varphi)(f)=T^*(\varphi)(e^{-i\theta}f) \\
    &=
    \re \varphi(T(e^{-i\theta}f))\le |\varphi(T(e^{-i\theta}f))| \\
    &\le
    \|\varphi\|\|T(e^{-i\theta}f)\|=\|\varphi\|\|e^{-i\theta}f\|=\|\varphi\|\|f\|.
\end{split}
\end{equation*}
Hence, we have 
$\|T^*(\varphi)\|\le \|\varphi\|$
for $\varphi\in B^*(H)$. We show that the inverse inequality $\|\varphi\|\le \|T^*(\varphi)\|$. Then, $\|T^*(\varphi)\|=\|\varphi\|$, so that $T^*$ is an isometry. Let $h\in B(H)$. Then,  $0\le\theta< 2\pi$ exists such that $\varphi(h)=|\varphi(h)|e^{i\theta}$. Then 
\begin{equation*}
    \begin{split}
        |\varphi(h)|&=e^{-i\theta}\varphi(h)=\varphi(e^{-i\theta}h) \\
        &=\varphi(T(T^{-1}(e^{-i\theta}h)))
        =\re\varphi(T(T^{-1}(e^{-i\theta}h))) \\
        &\le |T^*(\varphi)(T^{-1}(e^{-i\theta}h))|\le \|T^*(\varphi)\|\|T^{-1}(e^{-i\theta}h)\| \\
        &=\|T^*(\varphi)\|\|h\|.        
    \end{split}
\end{equation*}
As $h\in B(H)$ is arbitrary, we have that $\|\varphi\|\le \|T^*(\varphi)\|$. To prove that $T^*$ is surjective, 
we consider $(T^{-1})^*\colon B(G)^*\to B(H)^*$, similar to $T^*$; $(T^{-1})^*(\psi)(h)=\re \psi (T^{-1}(h))-i\re \psi (T^{-1}(ih))$ for $h\in B(H)$ and $\psi\in B(G)^*$. Let $\psi\in B(G)^*$ be arbitrary. 
Put $\varphi'=(T^{-1})^*(\psi)$. Then $\varphi'\in B(H)^*$.
For every $f\in B(G)$ we have that 
\begin{equation*}
\begin{split}
  T^*(\varphi')(f)&=\re \varphi'(T(f))-i\re \varphi'(T(if)) \\
  &=
  \re(\re\psi(T^{-1}(T(f)))-i\re \psi (T^{-1}(iT(f)))) \\
  &
  -i\re(\re \psi(T^{-1}(T(if)))-i\re\psi(T^{-1}(iT(if)))) \\
  &=
  \re\psi(f)-i\re\psi(if)=\psi(f).
\end{split}
\end{equation*}
As $\psi\in B(G)^*$ is arbitrary we infer that $T^*$ is surjective.

Let $\sigma(B(G))$ and $\sigma(B(H))$ be the spectra (the same as the maximal ideal spaces) of 
$B(G)$ and $B(H)$,  respectively. 
\begin{lemma}\label{multi}
We have 
\[
T^{*}(\sigma(B(H)))=\sigma(B(G)).
\]  
\end{lemma}
\begin{proof}
Let $\varphi \in \sigma(B(H))$. For any $f \in B(G)$, we have $\varphi(T(if))^2=-(\varphi(Tf))^2$. This implies that 
\begin{equation}\label{jyami0}
\text{$\varphi(T(if))=i\varphi(Tf)$ or $\varphi(T(if))=-i\varphi(Tf)$.}
\end{equation}
Let $f,g\in B(G)$ and $\varphi\in \sigma(B(H))$.
As $T$ and $\varphi$ are multiplicative and real linear we have
\begin{multline}\label{jyami1}
\varphi(T(if))\varphi(T(ig))=\varphi(T(if)T(ig))
\\
=\varphi(T(-fg))=-\varphi(T(fg))=(i\varphi(T(f)))(i\varphi(T(g))).
\end{multline}
If $\varphi(T(f))\ne 0$ and $\varphi(T(g))\ne0$, then $\eqref{jyami1}$ ensures that 
\begin{equation}\label{jyamira}
\text{$i\varphi(T(f))=\varphi(T(if))$ if and only if $i\varphi(T(g))=\varphi(T(ig))$}.
\end{equation}
Hence we have
\begin{equation}\label{jyami2}
\re i\varphi(T(f))\re i\varphi(T(g))=\re\varphi(T(if))\re\varphi(T(ig)).
\end{equation}
If $\varphi(T(f))=0$ or $\varphi(T(g))=0$, then we infer that $\varphi(T(if))=0$ or $\varphi(T(ig))=0$ respectively by \eqref{jyami0}. 
Hence, \eqref{jyami2} also holds for any $f,g \in B(G)$. Applying \eqref{jyami2}
we have 
\begin{equation*}
\begin{split}
&T^*(\varphi)(f)T^*(\varphi)(g) \\
&=
(\re\varphi(T(f))-i\re\varphi(T(if)))(\re \varphi(T(g))-i\re\varphi(T(ig))) \\
&=
\re\varphi(T(f))\re\varphi (T(g))-\re i\varphi(T(f))\re i\varphi(T(g)) \\
&
-i\left(\re\varphi(T(if))\re\varphi(T(g))-\re i\varphi(T(if))\re i\varphi(T(g))\right) \\
&=
\re\varphi(T(f))\re\varphi(T(g))-\im\varphi(T(f))\im \varphi(T(g))\\
&
-i\left(
\re\varphi(T(if))\re\varphi(T(g))-\im \varphi(T(if))\im \varphi(T(g))\right) \\
&=
T^*(\varphi(fg)).
\end{split}
\end{equation*}
Thus $T^{*}(\varphi)$ is a multiplicative linear functional. As $T^{*}$ is a surjective isometry and $\varphi \neq 0$, we obtain $T^{*}(\varphi) \neq 0$. Hence $T^*(\varphi)\in \sigma(B(G))$. To prove that $T^*$ maps $\sigma(B(H))$ onto $\sigma(B(G))$, 
we consider $(T^{-1})^*\colon B(G)^*\to B(H)^*$ in a similar way as $T^*$; $(T^{-1})^*(\psi)(h)=\re \psi (T^{-1}(h))-i\re \psi (T^{-1}(ih))$ for $h\in B(H)$ and $\psi\in B(G)^*$. Let $\psi\in \sigma(B(G))$ be arbitrary. As in the first part, we have $\varphi'=(T^{-1})^*(\psi)\in \sigma(B(H))$. For every $f\in B(G)$, we have 
\begin{equation*}
\begin{split}
  T^*(\varphi')(f)&=\re \varphi'(T(f))-i\re \varphi'(T(if)) \\
  &=
  \re(\re\psi(T^{-1}(T(f)))-i\re \psi (T^{-1}(iT(f)))) \\
  &
  -i\re(\re \psi(T^{-1}(T(if)))-i\re\psi(T^{-1}(iT(if)))) \\
  &=
  \re\psi(f)-i\re\psi(if)=\psi(f).
\end{split}
\end{equation*}
As $\psi\in \sigma(B(G))$ is arbitrary, we infer that $T^*$ maps $\sigma(B(H))$ onto $\sigma(B(G))$. 
\end{proof}

\begin{definition}
Let $f\in B(G)$ such that $T(f)$ is invertible in $B(H)$.  We define 
\[
\sigma(B(H))^{+}=\{ \varphi \in \sigma(B(H)) \mid \varphi(T(if))=i\varphi(T(f)) \}
\]
and 
\[
\sigma(B(H))^{-}=\{ \varphi \in \sigma(B(H)) \mid \varphi(T(if))=-i\varphi(T(f)) \}.
\]  
\end{definition}
Note that $\sigma(B(H))^{+}$ and $\sigma(B(H))^{-}$ are independent of the choice of $f\in B(G)$ by \eqref{jyamira}. 
This infers that for any $g \in B(G)$, if $\varphi \in \sigma(B(H))^{+}$ (resp. $\varphi \in \sigma(B(H))^{-}$), then $\varphi(T(ig))=i\varphi(T(g))$ (resp. $\varphi(T(ig))=-i\varphi(T(g)$). By \eqref{jyami0},  we have $\sigma(B(H))=\sigma(B(H))^{+}\cup \sigma(B(H))^{-}$, which is a disjoint union. 

\begin{lemma}\label{2}
We have 
\[
T^{*}(i \varphi)=
\begin{cases}
     i T^{*}(\varphi) \quad & \text{if $\varphi \in \sigma(B(H))^{+}$,} \\
    -iT^{*}(\varphi) \quad & \text{if $\varphi \in \sigma(B(H))^{-}$}.
  \end{cases}\\
\]
\end{lemma}
\begin{proof}
For any $f \in B(G)$, we have
\begin{equation*}
\begin{split}
T^{*}(i \varphi)(f)&=\begin{cases}
     \re \varphi(T(if)) + i \re\varphi(T(f))\quad & \text{if $\varphi \in \sigma(B(H))^{+}$,} \\
    -\re \varphi(T(if)) - i \re\varphi(T(f)) \quad & \text{if $\varphi \in \sigma(B(H))^{-}$,}
  \end{cases}\\
  &=\begin{cases}
      i T^{*}(\varphi)(f) \quad & \text{if $\varphi \in \sigma(B(H))^{+}$,} \\
    -iT^{*}(\varphi)(f) \quad & \text{if $\varphi \in \sigma(B(H))^{-}$}.
  \end{cases}
  \end{split}
\end{equation*}
\end{proof}

Note that there is a one-to-one correspondence between the unitary representations of $G$ and the nondegenerate $*$-representations of $C^{*}(G)$ (cf. Theorems 3.9 and 3.11 in \cite{folland2} and Proposition 2.7.4 \cite{Dixmier}). 
Thus we refer to the unitary representation $w$ of $G$, which is associated with the universal representation of $C^{*}(G)$, as the universal representation of $G$. 
Then it is proven in \cite[(2.10)]{Eym} that 
$w(x)$ is the point evaluation 
at $x$ for the universal representation $w$ of $G$ 
as follows: 
\[
w(x)(f)=f(x)
\]
for any $f \in B(G)$ and $x \in G$ (see also \cite[Remark 2.1.6.]{KL}).

 The spectrum $\sigma(B(G))$ of $B(G)$ is discussed in \cite[Lemma 3.2.3]{KL}. Some of these basic properties have been previously proven by Walter \cite[Theorem 1]{Wal}. Recall that the dual space $B(G)^*$ of $B(G)$ is identified with $W^*(G)$. Thus,  we can assume  that $\sigma(B(G))\subset W^*(G)$. In the following, for $\varphi_1, \varphi_2 \in \sigma(B(G))$, the multiplication $\varphi_1\varphi_2$ is defined as the element in the von Neumann algebra $W^*(G)$. Similarly, the involution $\varphi^*$ of $\varphi\in \sigma(B(G))$ is also defined in $W^*(G)$.

\begin{lemma}[Lemma 3.2.3 in \cite{KL}]\label{basic}
Let $G$ be a locally compact group.
\begin{enumerate}
\item If $\varphi_1, \varphi_2 \in \sigma(B(G))$, then $\varphi_1\varphi_2 \in \sigma(B(G))$ and ${\varphi_1}^{*} \in \sigma(B(G))$.
\item We have $w(G)=\{\varphi\in \sigma(B(G))\mid \text{$\varphi$ is a unitary in $W^*(G)$}\}$, where $w(G)=\{ w(x) \mid x \in G \}$.
\end{enumerate}
\end{lemma}

We note that $w(e_1) \in \sigma(B(G))$ (resp. $w(e_2) \in \sigma(B(H))$) is the unit of $W^*(G)$ (resp. $W^*(H)$) since $w$ is the universal representation of $G$ (resp. $H$). 
Since $T^{*}$ is a surjective real linear isometry from $W^{*}(H)$ onto $W^{*}(G)$, Proposition 2.1 in \cite{HW} (precisely it is described in the proof) ensures that $T^{*}(w(e_2))$ is a unitary element of $W^{*}(G)$. Hence, $b \in G$ such that $T^{*}(w(e_2))=w(b)$ according to Lemma \ref{basic} (2). By \cite[Proposition 2.1.]{HW}, there is a Jordan $*$-isomorphism $J\colon W^*(H)\to W^*(G)$ and a central projection $p \in W^{*}(G)$ such that 
\begin{equation}\label{jyusunomoto}
T^{*}(\varphi)=w(b)(pJ(\varphi)+(1-p)J(\varphi^{*})), \quad \varphi \in W^{*}(H).
\end{equation}

First, we present the following proposition. 
\begin{proposition}
\label{unital}
Let $G$ and $H$ be locally compact groups.  Suppose that  $T: B(G) \to B(H)$ is an isometric real algebra isomorphism with $T^{*}(w(e_2))=w(e_1)$. Then there exists a topological isomorphism or topological anti-isomorphism $\alpha \colon H \to G$ such that 
\[
T(f)(x)= f(\alpha(x)) \quad  f \in B(G),  \ x \in H,
\]
or 
\[
T(f)(x)= \overline{f(\alpha(x))} \quad  f \in B(G),  \ x \in H.
\]
\end{proposition}

\begin{remark}\label{remark1}
When $\alpha \colon H \to G$ is a group isomorphism, the map $x \mapsto \alpha(x^{-1})$ is an anti-group isomorphism, and vice versa.
\end{remark}

We will apply lemmas to prove Proposition \ref{unital}. While proving lemmas, we borrow from Walter's methods in \cite{Wal}. However, we require a more general argument because we consider real linear cases. Thus, we included this proof for convenience unless the statement could not be directly deduced from other literature.

From Lemma \ref{projection} to Lemma \ref{+or-} we assume that $T^{*}(w(e_2))=w(e_1)$. Hence, we have 
\begin{equation}\label{unitalequation}
    T^{*}(\varphi)=pJ(\varphi)+(1-p)J(\varphi^{*}), \quad \varphi \in W^{*}(H)
\end{equation}
by \eqref{jyusunomoto}.

\begin{lemma}\label{projection}
We have 
\[
T^{*}(\varphi)=
\begin{cases}
     pJ(\varphi) \quad & \text{if $\varphi \in \sigma(B(H))^{+}$,} \\
    (1-p)J(\varphi^{*}) \quad & \text{if $\varphi \in \sigma(B(H))^{-}$}.
  \end{cases}\\
\]
\end{lemma}
 \begin{proof}
Let $\varphi \in \sigma(B(H))^{+} $. By Lemma \ref{2}, we have
 \begin{multline*}
 ipJ(\varphi)+i(1-p)J(\varphi^{*})=iT^{*}(\varphi)=T^{*}(i\varphi)\\
 =pJ(i\varphi)+(1-p)J((i\varphi)^{*})=ipJ(\varphi)-i(1-p)J(\varphi^{*}).
 \end{multline*}
 Hence we obtain 
\begin{equation*}
(1-p) J(\varphi^{*})=0.
\end{equation*} 
This implies that $T^{*}(\varphi)=pJ(\varphi)$.  Let $\varphi \in \sigma(B(H))^{-} $. By Lemma \ref{2}, we get  
 \begin{multline*}
 -ipJ(\varphi)-i(1-p)J(\varphi^{*})=-iT^{*}(\varphi)=T^{*}(i\varphi)\\
 =pJ(i\varphi)+(1-p)J((i\varphi)^{*})=ipJ(\varphi)-i(1-p)J(\varphi^{*}).
 \end{multline*}
 Thus we obtain that $pJ(\varphi)=0$ and $T^{*}(\varphi)=(1-p)J(\varphi^{*})$. 

 \end{proof}
The next corollary is useful for determining whether a given $\varphi \in \sigma(B(H))$  belongs to $\sigma(B(H))^+$ or $\sigma(B(H))^-$. 
 \begin{corollary}\label{useful}
 Let $\varphi \in \sigma(B(H))$. 
     \begin{itemize}
	\item[(i)] $\varphi \in \sigma(B(H))^+$ $\Leftrightarrow$ $T^{*}(\varphi)=pJ(\varphi)$ $\Leftrightarrow$  $(1-p)J(\varphi^*)=0$.
     
	\item[(ii)] $\varphi \in \sigma(B(H))^-$ $\Leftrightarrow$ $T^{*}(\varphi)=(1-p)J(\varphi^*)$ $\Leftrightarrow$  $pJ(\varphi)=0$.
\end{itemize}
     
 \end{corollary}
\begin{proof}
 (i) Let $\varphi \in \sigma(B(H))^{+}$. Lemma \ref{projection} shows that $T^{*}(\varphi)=pJ(\varphi)$. If $T^{*}(\varphi)=pJ(\varphi)$ holds then (\ref{unitalequation}) implies that $(1-p) J(\varphi^{*})=0$. Finally, suppose that $(1-p) J(\varphi^{*})=0$. We prove that $\varphi\in \sigma(B(H))^+$. Suppose not. Then, $\varphi\in \sigma(B(H))^-$. Then by Lemma \ref{projection}, we have $T^*(\varphi)=(1-p)J(\varphi^*)=0$, as $T^*$ is a surjective real linear isometry,  we have that  $\varphi=0$ which is impossible. Thus, we obtain that $\varphi\in \sigma(B(H))^+$.\\
 (ii)  In a similar way to (i), we can prove the statement. 
\end{proof}

 \begin{lemma}\label{involution}
 If $\varphi  \in \sigma(B(H))^{+}$ then $\varphi^{*} \in \sigma(B(H))^{+}$. If $\varphi  \in \sigma(B(H))^{-}$ then $\varphi^{*} \in \sigma(B(H))^{-}$. 
 \end{lemma}
 \begin{proof}
Let $\varphi \in \sigma(B(H))^{+}$. By applying Lemma \ref{basic}  , we have $\varphi^{*} \in \sigma(B(H))$. Since $J$ is a Jordan $*$-isomorphism and $p$ is a central projection, we obtain $T^{*}(\varphi^{*})=pJ(\varphi^{*})+(1-p)J((\varphi^{*})^{*})=(pJ(\varphi)+(1-p)J(\varphi^{*}))^{*}=(T^{*}(\varphi))^{*}$. Lemma \ref{projection} states that 
\begin{equation}\label{projep}
T^{*}(\varphi^{*})=(T^{*}(\varphi))^{*}=(pJ(\varphi))^{*}=pJ(\varphi^{*}).
\end{equation} 
Then by Corollary \ref{useful} we have $\varphi^*\in \sigma(B(H))^+$.
A similar argument  yields $\varphi^{*} \in \sigma(B(H))^{-}$ for $\varphi\in \sigma(B(H))^-$. 

 \end{proof}

\begin{lemma}\label{alpha}
We have $J(w(y)w(x))=J(w(y))J(w(x))$ for any $x,y \in H$ or  $J(w(y)w(x))=J(w(x))J(w(y))$ for any $x,y \in H$.
\end{lemma}
\begin{proof}
We have $pT^{*}(\varphi)=pJ(\varphi)$ and $(1-p)T^{*}(\varphi)=(1-p)J(\varphi^{*})$ by (\ref{unitalequation}). Hence we have 
\[
J(\varphi)=pJ(\varphi)+(1-p)J(\varphi)=pT^{*}(\varphi)+(1-p)T^{*}(\varphi^{*}).
\]
Let $\varphi \in \sigma(B(H))$.  When $\varphi \in \sigma(B(H))^{+}$, Lemma \ref{projection} and Lemma \ref{involution} show that 
\begin{equation}\label{jyamira3}
J(\varphi)=ppJ(\varphi)+(1-p)pJ(\varphi^{*})=pJ(\varphi)=T^{*}(\varphi).
\end{equation}
When $\varphi \in \sigma(B(H))^{-}$, we also obtain 
\begin{equation}\label{jyamira4}
J(\varphi)=p(1-p)J(\varphi^{*})+(1-p)(1-p)J(\varphi)=(1-p)J(\varphi)=T^{*}(\varphi^{*}).
\end{equation}
Therefore, in either case, we obtain $J(\varphi) \in \sigma(B(G))$ by Lemma \ref{multi}. Moreover, $J(\sigma(B(H)))=\sigma(B(G))$.
In fact, let $\psi\in \sigma(B(G))$ be arbitrary. Then by Lemma \ref{multi}, $\varphi\in \sigma(B(H))$ exists with $T^*(\varphi)=\psi$. If $\varphi\in \sigma(B(G))^{+}$, then \eqref{jyamira3} asserts that $J(\varphi)=\psi$. Suppose that $\varphi\in \sigma(B(G))^{-}$. Then we have $\varphi^*\in \sigma(B(G))^{-}$ by Lemma \ref{involution}. Then by \eqref{jyamira4} we have $J(\varphi^*)=T^*(\varphi)=\psi$. 
By Lemma \ref{basic} we have $J(\varphi_1\varphi_2), 
J(\varphi_2\varphi_1), J(\varphi_1)J(\varphi_2), J(\varphi_2)J(\varphi_1)\in \sigma(B(G))$. As $J$ is a Jordan $*$-isomorphism, we have
\[
J(\varphi_1\varphi_2)+J(\varphi_2\varphi_1)=
J(\varphi_1)J(\varphi_2)+J(\varphi_2)J(\varphi_1).
\]
On the other hand, in general, it is well known that elements in the spectrum  $\sigma({\mathcal{A}})$ of 
a commutative Banach algebra $\mathcal{A}$ are linearly independent.
Applying this fact, we can prove by several calculations that
\begin{equation}\label{jya1}
\text{$J(\varphi_1\varphi_2)=J(\varphi_1)J(\varphi_2)$ or $J(\varphi_1\varphi_2)=J(\varphi_2)J(\varphi_1)$. }
\end{equation}

 Since any Jordan $*$-isomorphism  preserves unitary elements and $J|_{\sigma(B(H))}$ is a bijection from $\sigma(B(H))$ onto  $\sigma(B(G))$, 
$J|_{\sigma(B(H))}$ maps the unitary elements of $\sigma(B(H))$ onto the unitary elements of $\sigma(B(G))$. By Lemma \ref{basic},  this implies that $J|_{w(H)}: w(H) \to w(G)$, which is a bijection. By \cite[Theorem 10]{Kadison}, as $J:W^{*}(H) \to W^{*}(G)$ is a Jordan $*$-isomorphism between von Neumann algebras, there are central projections $Z_h \in W^{*}(H)$ and $Z_g \in W^{*}(G)$ such that $J|_{W^{*}(H)Z_h}: W^{*}(H)Z_h \to W^{*}(G)Z_g$ is an algebra $*$-isomorphism and $J|_{W^{*}(H)(1-Z_h)}: W^{*}(H)(1-Z_h) \to W^{*}(G)(1-Z_g)$ is an anti-algebra $*$-isomorphism.
Then, since $Z_h$ is a central projection, and a Jordan $*$ isomorphism preserves commutativity we have 
\begin{multline}\label{commutativeJ}
J(abZ_h)=J(aZ_hbZ_h)=J(aZ_h)J(bZ_h)\\
=J(a)J(Z_h)J(b)J(Z_h)=J(a)J(b)J(Z_h)
\end{multline}
for every $a,b\in W^*(H)$. Similarly, we have
\begin{equation}\label{anticommutativeJ}
    J(ab(1-Z_h))=J(b)J(a)J(1-Z_h)
\end{equation}
for every $a,b\in W^*(H)$.

For any $x \in H$, we define $A_x=\{ y \in H \mid (w(xy)-w(yx))Z_h=0\}$ and $B_x=\{ y \in H \mid (w(xy)-w(yx))(1-Z_h)=0 \}$. By a simple calculation, it is easy to check that $A_x$ and $B_x$ are subgroups of $H$. We prove that $A_x \cup B_x=H$.  Let $y\in H$ be arbitrary. Then by \eqref{jya1} we have $J(w(x)w(y))=J(w(x))J(w(y))$ or $J(w(y))J(w(x))$. Suppose first that $J(w(x)w(y))=J(w(x))J(w(y))$. Then we have 
\begin{equation*}
\begin{split}
    &J((w(x)w(y)-w(y)w(x))(1-Z_h))\\
    &=J(w(x)w(y)(1-Z_h))-J(w(y)w(x)(1-Z_h))\\
    &=J(w(x)w(y))J(1-Z_h)-J(w(x))J(w(y))J(1-Z_h)\\
    &=J(w(x))J(w(y))J(1-Z_h)-J(w(x))J(w(y))J(1-Z_h)=0\\
    \end{split}
    \end{equation*}
    since $J(w(x)w(y))=J(w(x))J(w(y))$ and (\ref{anticommutativeJ}) hold. 
Since $J$ is injective, we have that $y \in B_x$. Suppose contrary that $J(w(x)w(y))=J(w(y))J(w(x))$. Like in  the above discussion with (\ref{commutativeJ}), we see that $y\in A_x$. We conclude that $A_x\cup B_x=H$.

As $A_x$ and $B_x$ are subgroups of $H$, we obtain that either $A_x =H$ or $ B_x=H$.
Suppose not: $A_x\ne H$ and $B_x\ne H$. Thus, $y_1\in H\setminus A_x$ and $y_2\in H\setminus B_x$ exist. In this case $y_1\in B_x$ and $y_2\in A_x$ since $A_x\cup B_x=H$. Consider the element $y_1y_2\in H$. Then $y_1y_2\in A_x$ or $y_1y_2\in B_x$. Suppose that $y_1y_2\in B_x$. As $y_1^{-1}\in B_x$ we obtain $y_2=y_1^{-1}y_1y_2\in B_x$, which contradicts the choice of $y_2$. Similarly, we have a contradiction if we assume that $y_1y_2\in A_x$. We conclude that $A_x=H$ or $B_x=H$.

In addition as $A=\{ x \in H \mid A_x=H\}$ and $B=\{ x \in H \mid B_x=H\}$ are subgroups of $H$ and $A \cup B=H$, it must be either $A=H$ or $B=H$ according to a similar argument above.  When $A=H$, we have $J(w(x)w(y))=J(w(y))J(w(x))$ for any $x,y \in H$. 
To prove this, note that $w(x)w(y)Z_h=w(y)w(x)Z_h$. By (\ref{commutativeJ}) and (\ref{anticommutativeJ}), we have 
\begin{equation*}
\begin{split}
    J(w(x)w(y))&=J(w(x)w(y))J(Z_h)+J(w(x)w(y))J(1-Z_h)\\
&=J(w(x)w(y)Z_h)+J(w(x)w(y)(1-Z_h))\\
&=J(w(y)w(x)Z_h)+J(w(x)w(y)(1-Z_h))\\
&=J(w(y))J(w(x))J(Z_h)+J(w(y))J(w(x))J(1-Z_h)\\
&=J(w(y))J(w(x)).
\end{split}
\end{equation*}
By a similar argument, we conclude that $B=H$ implies that  $J(w(x)w(y))=J(w(x))J(w(y))$ for any $x,y \in H$. 
\end{proof}

\begin{definition}
We say that $J$ is multiplicative $($resp. anti-multicative$)$ if $J(w(x)w(y))=J(w(x))J(w(y))$ $($resp. $J(w(x)w(y))=J(w(y))J(w(x)) )$ for every pair $x,y\in H$. 
We define $\alpha:H \to G$ by 
$\alpha(x)=w^{-1}\circ J\circ w (x)$, $x\in H$.
\end{definition}

Note that $\alpha$ is a group isomorphism if $J$ is multiplicative and an anti-group isomorphism if $J$ is anti-multiplicative. 

\begin{lemma}\label{pointwise}
If $w(x) \in \sigma(B(H))^{+}$ then $T(f)(x)=f(\alpha(x))$. 
If $w(x) \in \sigma(B(H))^{-}$ then $T(f)(x)=\overline{f(\alpha(x^{-1}))}$. 
\end{lemma}
\begin{proof}
For any $f \in B(G)$, note that 
\begin{multline*}
f(\alpha(x))=w(\alpha(x))(f)=J(w(x))(f)\\
=pT^{*}(w(x))(f)+(1-p)T^{*}(w(x)^{*})(f). 
\end{multline*}
If $w(x) \in \sigma(B(H))^{+}$ then we have by Lemma \ref{2} and Lemma \ref{projection} that 
\begin{equation*}
\begin{split}
f(\alpha(x))&=pT^{*}(w(x))(f)+(1-p)T^{*}(w(x)^{*})(f)\\&=ppJ(w(x))(f)+(1-p)pJ(w(x)^{*})(f)\\
&=pJ(w(x))(f)=T^{*}(w(x))(f)= \re w(x)(T(f))-i \re w(x)(T(if))
\end{split}
\end{equation*}
and
\begin{equation*}
\begin{split}
if(\alpha(x))&=iT^{*}(w(x))(f)=T^{*}(iw(x))(f)\\
&=\re iw(x)(T(f))-i \re iw(x)(T(if))\\
&=-\im w(x)(T(f))+i \im w(x)(T(if)).
\end{split}
\end{equation*}
Therefore we have 
\begin{multline*}
f(\alpha(x))=\re f(\alpha(x))+i \im f(\alpha(x))\\
=\re T(f)(x)+i \im T(f)(x)=T(f)(x)
\end{multline*}
for any $f \in B(G)$ and $x \in H$ with $w(x)\in \sigma(B(H))^{+}$.
If  $w(x) \in \sigma(B(H))^{-}$ then we have by Lemma \ref{2} and Lemma \ref{projection} that 
\begin{equation*}
\begin{split}
f(\alpha(x))&=pT^{*}(w(x))(f)+(1-p)T^{*}(w(x)^{*})(f)\\&=p(1-p)J(w(x)^{*})(f)+(1-p)(1-p)J(w(x))(f)\\
&=T^{*}(w(x)^{*})(f)= \re w(x^{-1})(T(f))-i \re w(x^{-1})(T(if))
\end{split}
\end{equation*}
and
\begin{equation*}
\begin{split}
-if(\alpha(x))&=-iT^{*}(w(x)^{*})(f)=T^{*}(iw(x^{-1}))(f)\\
&=\re iw(x^{-1})(T(f))-i \re iw(x^{-1})(T(if))\\
&=-\im w(x^{-1})(T(f))+i \im w(x^{-1})(T(if)).
\end{split}
\end{equation*}
We obtain that 
\begin{multline}
  \overline{f(\alpha(x))}=\re f(\alpha(x))-i \im f(\alpha(x)) \\
  =
  \re T(f)(x^{-1})+i \im T(f)(x^{-1})=T(f)(x^{-1}).  
\end{multline}
\end{proof}

\begin{lemma}\label{+or-}
Suppose that $w(x),w(y)\in \sigma(B(H))^+$ $($resp. $\sigma(B(H))^-$$)$. Then $w(xy)\in \sigma(B(H))^+$ $($resp. $\sigma(B(H))^-$$)$. We have a dichotomy 
 $w(H)\subset\sigma(B(H))^+$ or $w(H)\subset \sigma(B(H))^-$.
\end{lemma}
\begin{proof}
For any $w(x), w(y) \in \sigma(B(H))^{+}$, we have
\begin{equation*}
\begin{split}
&(1-p)J(w(xy)^*)
=(1-p)J(w(y)^{*}w(x)^{*})\\
&=\begin{cases}
     (1-p)J(w(y)^{*})(1-p)J(w(x)^{*}) \quad & \text{if $J$ is multiplicative,} \\
    (1-p)J(w(x)^{*})(1-p)J(w(y)^{*})\quad & \text{if $J$ is anti-multiplicative}
  \end{cases}\\
&=0
\end{split}
\end{equation*}
by Corollary \ref{useful}. Applying Corollary \ref{useful} again, we conclude that $w(xy)\in \sigma(B(H))^+$.  
Note that $w(x)^*=w(x^{-1})$ for $x\in \sigma(B(H))$. Hence
if $x\in  H$ with $w(x)\in \sigma(B(H))^+$, then  $w(x)^*\in\sigma(B(H))^+$ by Lemma \ref{involution}, so that $w(e_2)=w(xx^{-1})=w(x)w(x)^*\in \sigma(B(H))^+$ .
By a similar argument, we obtain that $w(xy) \in \sigma(B(H))^{-}$ for any $w(x), w(y) \in \sigma(B(H))^{-}$.
Thus, the existence of an element $x\in H$ with $w(x)\in \sigma(B(H))^-$ similarly ensures that $w(e_2)\in \sigma(B(H))^-$.
We conclude that $w(e_2)\in \sigma(B(H))^+$ or $w(e_2)\in \sigma(B(H))^-$. 
As $\sigma(B(H))^+\cap \sigma(B(H))^-=\emptyset$, we conclude that either $w(H)\subset \sigma(B(H))^+$ or $w(H)\subset \sigma(B(H))^-$ holds.

\end{proof}

Now we prove Proposition \ref{unital}.
\begin{proof}[Proof of Proposition \ref{unital}]
By Lemma \ref{+or-} we have a dichotomy $w(H) \subset \sigma(B(H))^{+}$ or $w(H) \subset \sigma(B(H))^{-}$. 
When $w(H) \subset \sigma(B(H))^{+}$, by Lemma \ref{pointwise} we have
\[
T(f)(x)=f(\alpha(x)), \quad x \in H.
\]
When  $w(H) \subset \sigma(B(H))^{-}$, by Lemma \ref{pointwise} we have
\[
T(f)(x)=\overline{f(\alpha(x^{-1}))}, \quad x \in H.
\]
By Remark \ref{remark1},  we note that $x \mapsto \alpha(x^{-1})$ is also a group isomorphism or an anti-group isomorphism. 

Let $U$ be an open neighborhood of $e_1$ on $G$. There is a continuous positive definite function $f_0$ on $G$, which has compact support,  such that $0 \le f_0 \le 1$ with $f_0(e_1)=1$ and $f_0=0$ on $G \setminus U$. For a net $\{x_{\beta}\}$ with  $x_{\beta} \to e_2$, we have $f_0(\alpha(x_{\beta})) \to f_0(\alpha(e_2))$ as $T(f_0) \in B(H)$. This implies that $\alpha(x_{\beta}) \to \alpha(e_2)=e_1$. Hence $\alpha$ is continuous. By applying a similar argument to $\alpha^{-1}$, we obtain that $\alpha$ is a homeomorphism. Thus we complete the proof.
\end{proof}
Finally, we characterize isometric real algebra isomorphisms on Fourier-Stieltjes algebras. Theorem \ref{realisometryonBG} is a real-linear-counterpart of the theorem of Walter \cite[Theorem 2]{Wal}.

\begin{theorem}\label{realisometryonBG}
Let $G$ and $H$ be locally compact groups.  The map $T: B(G) \to B(H)$ is an isometric real algebra isomorphism if and only if there is a topological isomorphism or a topological anti-isomorphism $\alpha: H \to G$ and $b \in G$ such that 
\[
T(f)(x)= f(b \alpha(x)) \quad  f \in B(G),  \ x \in H,
\]
or 
\[
T(f)(x)= \overline{f(b \alpha(x))} \quad  f \in B(G),  \ x \in H.
\]
\end{theorem}
\begin{proof}
If $T: B(G) \to B(H)$ takes the form $T(f)(x)= f(b \alpha(x))$, where $\alpha: H \to G$ is either a topological group isomorphism or a topological anti-group isomorphism, then \cite[Corollary of Theorem 3]{Wal} shows that $T$ is an isometric complex algebra isomorphism from $B(G)$ onto $B(H)$.

 Similarly, if $T(f)(x)= \overline{f(b \alpha(x))}$, then $T$ is an isometric conjugate-linear  algebra isomorphism from $B(G)$ onto $B(H)$.

We prove the converse statement. 
Suppose that $T: B(G) \to B(H)$ is an isometric real algebra isomorphism. According to \eqref{jyusunomoto}, $b \in G$ such that $T^{*}(w(e_2))=w(b)$.

  We define $T_r:B(G) \to B(G)$ by $T_r(f)(x)=f(b^{-1}x)$. Since the map $x \mapsto b^{-1}x$ is an affine map on $G$, $T_r$ is an isometric isomorphism on $B(G)$ as in the same argument as the first part of the proof (cf. \cite{Wal}). It follows that  $T \circ T_r\colon B(G) \to B(H)$ is an isometric real algebra isomorphism. Then we obtain the following equation:  
\begin{equation*}
\begin{split}
(T \circ T_r)^{*}(w(e_2))f&=\re w(e_2) (T \circ T_r(f))-i \re w(e_2) (T \circ T_r(if))\\
&=T^{*}(w(e_2))(T_r(f))=w(e_1)f
\end{split}
\end{equation*}
for any $f \in B(G)$. Thus 
\[
(T \circ T_r)^{*}(w(e_2))=w(e_1).
\]
By applying Proposition \ref{unital} to $T \circ T_r\colon B(G) \to B(H)$, we obtain  a topological group isomorphism or a topological anti-group isomorphism $\alpha\colon H \to G$ such that 
\[
T \circ T_r(f)(x)=f(\alpha(x)), \quad x \in H.
\]
or 
\[
T \circ T_r(f)(x)=\overline{f(\alpha(x))}, \quad x \in H.
\]
Thus, for any $f \in B(G)$ and $ x \in H$, we obtain the desired forms 
\[
Tf(x)=T\circ T_r \circ T_r^{-1}(f)(x)=T_r^{-1}(f)(\alpha(x))=f(b \alpha(x)), 
\]
or 
\[
Tf(x)=T\circ T_r \circ T_r^{-1}(f)(x)=\overline{T_r^{-1}(f)(\alpha(x))}=\overline{f(b\alpha(x))}.
\]

\end{proof}

We also characterize isometric real algebra isomorphisms on Fourier algebras $A(G)$. For a locally compact group $G$, $A(G)$ is the space of all matrix coefficients of the left regular representation $\lambda_G$ from $G$ into the unitary operators on $L^{2}(G)$. 

\begin{theorem}\label{realisometryonAG}
Let $G$ and $H$ be locally compact groups.  The map $T \colon A(G) \to A(H)$ is an isometric real algebra isomorphism if and only if there is a topological isomorphism or a topological anti-isomorphism $\alpha \colon H \to G$ and $b \in G$ such that 
\[
T(f)(x)= f(b \alpha(x)) \quad  f \in A(G),  \ x \in H,
\]
or 
\[
T(f)(x)= \overline{f(b \alpha(x))} \quad  f \in A(G),  \ x \in H.
\]
\end{theorem}
\begin{proof}
 We note that the spectrum $\sigma(A(G))$ of $A(G)$ is homeomorphic to $G$ \cite[Theorem 2.3.8]{KL}. The dual space of $A(G)$ is isometrically  
 isomorphic as Banach space to $VN(G)$, where $VN(G)$ is the von Neumann algebra generated by $\lambda_G$. Any isometric real algebra isomorphism $T: A(G) \to A(H)$ induces a surjective real linear isometry $T^{*}\colon VN(H) \to VH(G)$ by a similar formula as \eqref{dualmap}. 
 We prove the theorem by an argument essentially identical to that given in Theorem \ref{realisometryonBG}, replacing $W^{*}(G)$ and $W^{*}(H)$ by $VH(G)$ and $VN(H)$ respectively.
\end{proof}

\section{isometries between groups of invertible elements}


The study in \cite{H1} initially focused on isometries between groups of invertible elements in Banach algebras. Later, this work was generalized and referenced in \cite{H2}.
In the case where $G$ and $H$ are locally compact {\it abelian} groups, $B(G)$ and $B(H)$ represent the Fourier-Stieltjes transform of the measure algebra $M(\widehat{G})$ and $M(\widehat{H})$ for their respective dual groups $\widehat{G}$ and $\widehat{H}$, respectively. It is worth noting that $B(G)$ and $B(H)$ are isometrically isomorphic to $M(\widehat{G})$ and $M(\widehat{H})$, respectively.
Suppose that $T_0$ is a surjective isometry between open subgroups $\mathfrak{A}$ and $\mathfrak{B}$ of the groups of invertible elements in $B(G)$ and $B(H)$, respectively. Applying 
\cite[Theorem 3.1]{H3}, $\widehat{G}$ and $\widehat{H}$ are topologically isomorphic. This implies that $G$ and $H$ are topologically equivalent through the Pontryagin duality theorem.
Theorem \ref{main} generalizes this result to Fourier-Stieltjes algebras on locally compact groups.

We denote the group of all invertible elements in the Fourier-Stieltjes algebra $B(G)$ by $B(G)^{-1}$.

\begin{theorem}\label{main}
Let $G$ and $H$ be locally compact groups.   Let $\mathfrak{A}$ and $\mathfrak{B}$ be open subgroups of $B(G)^{-1}$ and $B(H)^{-1}$, respectively.  Suppose that $T_0$ is a surjective isometry from $\mathfrak{A}$ onto $\mathfrak{B}$. Then $T_0$ is extended to an isometric real algebra isomorphism from $B(G)$ onto $B(H)$ and there is a topological isomorphism or a topological anti-isomorphism $\alpha:H \to G$ and $b \in G$ such that 
\[
T_0(f)(x)= T_0(1)f(b \alpha(x)) \quad  f \in \mathfrak{A},  \ x \in H,
\]
or 
\[
T_0(f)(x)= T_0(1)\overline{f(b \alpha(x))} \quad  f \in \mathfrak{A},  \ x \in H.
\]
\end{theorem}

\begin{proof}
First, we note that $B(G)$ and $B(H)$  are unital semisimple commutative Banach algebras. 
    Applying \cite[Theorem 3.3]{H1} or \cite[Corollary 5.2]{H2} for $A=B(G)$ and $B=B(H)$ we have 
    that $(T_0(1))^{-1}T_0$ extends to an isometric real algebra isomorphism $T$ from $A$ onto $B$. By Theorem \ref{realisometryonBG} we obtain a topological isomorphism or a topological anti-isomorphism $\alpha\colon H\to G$ and $b\in G$  such that 
\[
T(f)(x)= f(b \alpha(x)) \quad  f \in B(G),  \ x \in H,
\]
or 
\[
T(f)(x)= \overline{f(b \alpha(x))} \quad  f \in B(G),  \ x \in H.
\]
As $T$ is an extension of $(T_0(1))^{-1}T_0$ we have the desired equalities.
\end{proof}
We apply \cite[Theorem 3.3]{H1} or \cite[Corollary 5.2]{H2} in the proof above. 

\begin{corollary}\label{topologicalygroupiso}
    For locally compact groups $G$ and $H$, suppose that open subgroups $\mathfrak{A}$ and $\mathfrak{B}$ of $B(G)^{-1}$ and $B(H)^{-1}$ are isometric. Then $G$  and $H$ are topologically isomorphic to each other.
\end{corollary}
\begin{proof}
Theorem \ref{main} yields either a topological isomorphism of $H$ onto $G$, or a topological anti-isomorphism of $H$ onto $G$. In the latter case, following the anti-isomorphism with the inverse map $x\in G \mapsto x^{-1}\in G$ yields a topological isomorphism of $H$ onto $G$ as Remark \ref{remark1}.
\end{proof}

\subsection*{Acknowledgments}
The first author was supported by 
JSPS KAKENHI Grant Numbers JP19K03536. 
The second author was supported by JSPS KAKENHI Grant Numbers JP21K13804. 


\begin{thebibliography}{99}


\bibitem{Co} P.~J.~Cohen, On homomorphisms of the group algebras, \emph{Amer. J. Math.} {\bf 82} (1960), 213--226.

\bibitem{Daws} M.~Daws, Completely bounded homomorphisms of the Fourier algebra revisited, \emph{J. Group Theory} {\bf 25} (2022), no.3, 579--600. 

\bibitem{Dixmier} J.~Dixmier, $C^{*}$-algebras, North-Holland Mathematical Library, Vol. 15. Amsterdam-New York-Oxford, North-Holland Publishing Co., 1977.

\bibitem{Eym} P. Eymard, L'alg\`{e}bre de Fourier d'un groupe localement compact, \emph{Bull. Soc. Math. France} {\bf 92} (1964), 181--236.

\bibitem{folland} G.~B.~Folland, Real Analysis, 
Pure and Applied Mathematics. A Wiley-Interscience Series of Texts, Monographs, and Tracts. New York, NY: Wiley. (1999).

\bibitem{folland2}
G.~B.~Foland, A course in abstract harmonic analysis, 2nd updated edition, 
Textbooks in Mathematics, Boca Raton, FL: CRC Press. (2016)

\bibitem{H1} O.~Hatori, Isometries between groups of invertible elements in Banach algebras, \emph{Studia.Math.} {\bf 194} (2009), no.3, 293--304.

\bibitem{H2} O.~Hatori, Algebraic properties of isometries between groups of invertible elements in Banach algebras, \emph{J. Math. Anal. Appl. } {\bf 376} (2011), no.1, 84--93.

\bibitem{H3} O.~Hatori, New criteria for equivalence of locally compact abelian groups, \emph{J. Group Theory} {\bf 15} (2012), no.2, 271--277.

\bibitem{HW} O.~Hatori and K.~Watanabe, Isometries between groups of invertible elements in $C^{*}$-algebras, \emph{Studia.Math.} {\bf 209} (2012), no.2, 103--106.

\bibitem{HS} E.~Hewitt and K.~Stromberg, Real and abstract analysis. A modern treatment of the theory of functions of a real variable. 2nd printing corrected, Berlin-Heidelberg-New York: Springer-Verlag. VIII, (1969).

\bibitem{Il} M.~Ilie, On Fourier algebra homomorphisms, \emph{J. Funct. Anal.} {\bf 213} (2004), 88--110.



\bibitem{IlieSpronk} M.~Ilie and N.~Spronk, Completely bounded homomorphisms of the Fourier algebras, \emph{J. Funct. Anal.} {\bf 225} (2005), no.2, 480--499.


\bibitem{KL} E.~Kaniuth and A.~T.-M.~Lau, Fourier and Fourier-Stieltjes Algebras on Locally Compact groups, Mathematical Surveys and Monographs 231, American Mathematical Society, Providence, Rhode Island, (2018). 

\bibitem{Kadison} R.~Kadison, Isometries of operator algebras, \emph{Ann. of Math.} (2) {\bf 54} (1951), 325–338.

\bibitem{LePham} H.~Le Pham, Contractive homomorphisms of Fourier algebras, \emph{Bull. London Math. Soc.} {\bf 42} (2010), 937--947. 

\bibitem{spronk} N.~Spronk, Dual space of $L^1$, 

\url{https://www.math.uwaterloo.ca/~nspronk/math451/Lone_dual.pdf}


\bibitem{takesaki} M.~Takesaki, Theory of operator algebras. I, 
Springer-Verlag, New York-Heidelberg, (1979).

\bibitem{Wal} M.~E.~Walter, W*-algebras and nonabelian harmonic analysis, \emph{J. Funct. Anal.} {\bf 11} (1972), 17--38.

\end{thebibliography}
\end{document}